\documentclass[12pt]{amsart}
\usepackage{enumitem}
\usepackage{amsaddr}
\usepackage{amssymb}
\usepackage{tikz}
\usetikzlibrary{calc,tikzmark}
\usepackage{diagbox}
\usepackage{epic}

\calclayout

\newcommand{\CC}{\mathbb{C}}

\newcommand{\KK}{\mathbb{K}}

\newtheorem{theorem}{Theorem}[section]

\newtheorem{problem}[theorem]{Problem}

\theoremstyle{definition}
\newtheorem{definition}[theorem]{Definition}

\theoremstyle{remark}

\theoremstyle{proposition}

\theoremstyle{corollary}

\theoremstyle{conjecture*}
\newtheorem*{conjecture*}{Conjecture}

\numberwithin{equation}{section}

\begin{document}

\title[The Lie algebra of vector fields on $\KK^n$ is $2$-generated]{The Lie algebra of polynomial vector fields on the affine space is $2$-generated}


\author{Ivan Beldiev}
\address{HSE University, Faculty of Computer Science, Pokrovsky Boulevard 11, Moscow, 109028 Russia}
\email{ivbeldiev@gmail.com, isbeldiev@hse.ru}
\thanks{The study was implemented in the framework of the Basic Research Program at the HSE University.}



\date{}

\dedicatory{}

\begin{abstract}
    We prove that the infinite-dimensional Lie algebra of polynomial vector fields on the affine space $\KK^n$ is generated by two explicitly given elements.
\end{abstract}

\maketitle

\section{Introduction}

We assume that the ground field $\KK$ is algebraically closed of characteristic zero. The Lie algebras and algebraic varieties in this note are defined over $\KK$.

Suppose that we have a Lie algebra $\mathfrak g$ and a subset $S \subseteq \mathfrak g$. By definition, the minimal (with respect to inclusion) Lie subalgebra of $\mathfrak g$ containing $S$ is called the Lie subalgebra generated by $S$. We denote this subalgebra by $\textnormal{Lie}(S)$. In other words, $\textnormal{Lie}(S)$ consists of all elements that can be obtained from the elements of $S$ by a finite composition of Lie brackets and linear combinations. We say that $\mathfrak g$ is generated by $S\subseteq \mathfrak g$ if $\textnormal{Lie}(S) = \mathfrak g$.

For a given Lie algebra $\mathfrak g$, a natural question to ask is what is the minimal number of elements that generate $\mathfrak g$. Let us say that $\mathfrak g$ is $k$-generated if $\mathfrak g$ can be generated by $k$ elements. Since the Lie bracket of any element with itself vanishes, $k$ cannot be equal to~$1$ unless $\mathfrak g$ has dimension $1$ or $0$ and the Lie bracket on $\mathfrak g$ is identically zero.

However, the question whether $\mathfrak g$ can be generated by two elements is non-trivial. For example, it was proved by Kuranishi that all finite-dimensional semisimple Lie algebras are $2$-generated (see \cite[Section 2]{Kur}). Moreover, it was proved later by Ionescu (see \cite{Ion}) that any finite-dimensional simple Lie algebra $\mathfrak g$ over $\mathbb{C}$ satisfies the so-called 1.5-generation property: for any non-zero $x\in \mathfrak g$ there is an element $y\in\mathfrak g$ such that $\mathfrak g$ is generated by $x,y$. In fact, the same result holds if we require $x,y$ to be nilpotent elements for the special linear algebra $\mathfrak{sl}(n,\KK)$ and the symplectic algebra $\mathfrak{sp}(2n, \KK)$ as was shown by Chistopolskaya (\cite{Chist, Chist2}). These and other related questions (in particular, for positive characteristic) were also studied in other works, see, for example, \cite{Bois, Levy, Rich}.

Another natural question is whether a given Lie algebra $\mathfrak g$ can be generated by a finite number of elements. Clearly, any finite-dimensional Lie algebra is finitely generated, so this question makes sense only for infinite-dimensional Lie algebras. An important class of such algebras is the Lie algebras of algebraic vector fields on affine varieties. For an affine algebraic variety $X$, this Lie algebra is denoted by $\textnormal{Der}(X)$ and consists of all $\KK$-derivations of the algebra $\KK[X]$ of regular functions on $X$, i.e., of all $\KK$-linear maps $D\colon \KK[X] \to \KK[X]$ such that $D(fg) = D(f)g + fD(g)$ for any $f,g\in\KK[X]$. In coordinate terms, if $X$ is embedded in the affine space $\mathbb{A}^n$ as a closed subvariety and $I\subseteq\KK[X]$ is the ideal of all polynomials vanishing on $X$, then any algebraic vector field $D\in\textnormal{Der}(X)$ can be written in the form
$$D = g_1\frac{\partial}{\partial z_1} + g_2\frac{\partial}{\partial z_2} + \ldots + g_n\frac{\partial}{\partial z_n},$$
where $z_1, z_2, \ldots, z_n$ are affine coordinates on $\mathbb{A}^n$ and $g_1, g_2, \ldots, g_n$ are regular functions on~$X$ such that the functions
$$g_1\frac{\partial f}{\partial z_1} + g_2\frac{\partial f}{\partial z_2} + \ldots + g_n\frac{\partial f}{\partial z_n}$$
vanish on $X$ for any $f\in I$ (see, for example, \cite[Proposition 3.1]{BFu}).

The problem of finite generation of these Lie algebras for concrete affine varieties is studied in recent works by Andrist for $\KK = \CC$, see \cite{Andr, Andr2, Andr3}. In \cite{Andr}, it is proved that the Lie algebra $\textnormal{Der}(\CC^n)$ is $3$-generated and three generators of $\textnormal{Der}(\CC^n)$ are given explicitly. The proof in fact works not only for the field of complex numbers but for any algebraically closed field of characteristic zero. In \cite{Andr2, Andr3}, the author considers other affine varieties, namely the smooth quadric $SL(2,\CC)$, the quadratic cone $\{(x,y,z)\in\CC^3\mid xy = z^2\}$ and non-singular Danielewski surfaces. It turns out that the Lie algebras of algebraic vector fields on these varieties are also finitely generated by a relatively small number of elements and the corresponding sets of generators are written explicitly. Moreover, the vector fields given in \cite{Andr, Andr2, Andr3} generating these Lie algebras are complete, i.e. their flow maps exist for all complex times. This allows to find a finite-number of one-parameter subgroups in the group of holomorphic automorphisms of these varieties that generate a subgroup acting infinitely transitively on them.

The aim of this note is to prove that the Lie algebra $\textnormal{Der}(\KK^n)$ is in fact $2$-generated. Namely, it is generated by the following two vector fields:

$$U = \frac{\partial}{\partial z_n},$$
$$V  = z_n^{4n}\frac{\partial}{\partial z_{n-1}} + (z_nz_{n-1})^{4n-4}\frac{\partial}{\partial z_{n-2}} + \ldots + (z_nz_{n-1}\ldots z_2)^{8}\frac{\partial}{\partial z_1} + (z_nz_{n-1}\ldots z_1)^4\frac{\partial}{\partial z_n}.$$

We prove it in Section 2 showing explicitly how to express any algebraic vector field on $\KK^n$ via $U$ and $V$ using Lie brackets and linear combinations.

In Section 3, we say a few words about complete vector fields and infinite transitivity. We show that our vector field $V$ is not complete, so the question whether $\textnormal{Der}(\CC^n)$ can be generated by two complete algebraic vector fields remains open.

\section{Acknowledgements}

The author is grateful to his academic supervisor Ivan Arzhantsev for posing the problem and useful discussions.

\section{The main result}

Consider the Lie algebra $\textnormal{Der}(\KK^n)$ of algebraic (in other words, polynomial) vector fields on the affine space $\KK^n$ with coordinates $z_1, z_2, \ldots, z_n$. Explicitly,
$$\textnormal{Der}(\KK^n) = \left\{ f_1\frac{\partial}{\partial z_1} + f_2\frac{\partial}{\partial z_2} + \ldots + f_n\frac{\partial}{\partial z_n}\,\middle|\, f_1, f_2, \ldots, f_n\in\KK[z_1, z_2, \ldots, z_n] \right\},$$
and the Lie bracket is given by
$$\left[ f\frac{\partial}{\partial z_i}, g\frac{\partial}{\partial z_j}\right] = f\frac{\partial g}{\partial z_i}\frac{\partial}{\partial z_j} - g\frac{\partial f}{\partial z_j}\frac{\partial}{\partial z_i}$$
for any $f,g\in\KK[z_1,z_2,\ldots, z_n]$ and for any $1\leq i,j\leq n$.
The Lie algebra $\textnormal{Der}(\KK^n)$ is commonly denoted by $W_n$.

In this section, we prove that $W_n$ can be generated by two explicitly given elements. Our main result is the following theorem.

\begin{theorem}\label{main}
    The Lie algebra of polynomial vector fields on $\KK^n$ is generated by the following two elements:
$$ U = \frac{\partial}{\partial z_n},$$
$$V  = z_n^{4n}\frac{\partial}{\partial z_{n-1}} + (z_nz_{n-1})^{4n-4}\frac{\partial}{\partial z_{n-2}} + \ldots + (z_nz_{n-1}\ldots z_2)^{8}\frac{\partial}{\partial z_1} + (z_nz_{n-1}\ldots z_1)^4\frac{\partial}{\partial z_n}.$$
\end{theorem}

\begin{proof} Denote by $L$ the Lie algebra generated by $U$ and $V$. We need to prove that $L$ coincides with $W_n$.
    \begin{enumerate}
        \item Let us start acting by $\textnormal{ad}_U$ on $V$ (in other words, this means applying $\frac{\partial}{\partial z_n}$ to each coefficient of~$V$). After acting $4n-3$ times, we obtain, up to a non-zero coefficient, the vector feld $z_n^3\frac{\partial}{\partial z_{n-1}}$. After we continue acting by $\textnormal{ad}_U$ on this vector field, we see that $z_n^s\frac{\partial}{\partial z_{n-1}} \in L$ for any $0 \leq s \leq 3$.
        \item Now, suppose by induction that $z_n^{s_n}z_{n-1}^{s_{n-1}}\ldots z_{i+1}^{s_{i+1}}\frac{\partial}{\partial z_{i}} \in L$ for any $i\geq k$ and for any $0 \leq s_{n}, s_{n-1}, \ldots, s_{i+1} \leq 3$. In particular, $\frac{\partial}{\partial z_{k}} \in L$. Let us start acting by $\textnormal{ad}_{\partial/\partial z_k}$ on~$V$. After we do this $4k - 3$ times, we obtain, up to a non-zero coefficient, the vector field $(z_nz_{n-1}\ldots z_{k+1})^{4k}z_k^3\frac{\partial}{\partial z_{k-1}}$. Acting by $\textnormal{ad}_{\partial/\partial z_j}$, $j\geq k$, on this vector field sufficiently many times, we obtain $$z_n^{s_n}z_{n-1}^{s_{n-1}}\ldots z_{k}^{s_k}\frac{\partial}{\partial z_{k-1}}$$ for any $0 \leq s_{n}, s_{n-1}, \ldots, s_{k} \leq 3$. In particular, $\frac{\partial}{\partial z_{k-1}} \in L$. At the $n$-th step of this procedure, we obtain $$z_n^{s_n}z_{n-1}^{s_{n-1}}\ldots z_{1}^{s_1}\frac{\partial}{\partial z_{n}}$$ for any $0 \leq s_{n}, s_{n-1}, \ldots, s_1 \leq 3$.
        \item For any $k < n$, we have
        $$\left[z_k^3\frac{\partial}{\partial z_n}, z_n\frac{\partial}{\partial z_k}\right] = z_k^3\frac{\partial}{\partial z_k} - 3z_k^2z_n\frac{\partial}{\partial z_n}.$$
        Since we already proved that $z_k^2z_n\frac{\partial}{\partial z_n} \in L$, this implies that $z_k^3\frac{\partial}{\partial z_k}\in L$. Acting by $\textnormal{ad}_{\partial/\partial z_k}$ on this vector field, we see that $z_k^s\frac{\partial}{\partial z_k} \in L$ for any $k$ and for any $0 \leq s\leq 3$ (for $k = n$, it was proved at the previous step).
        \item For any $i,j$ such that $i\ne j$, we have
        $$\left[z_i\frac{\partial}{\partial z_n}, z_n\frac{\partial}{\partial z_j}\right] = z_i\frac{\partial}{\partial z_j}.$$
        This proves that $z_i\frac{\partial}{\partial z_j}\in L$ for any $i, j$ from $1$ to $n$ (for $i = j$, it was proved at the previous step).
        \item For any $i$ and for any non-negative integer $s$, we have
        $$\left[z_i^s\frac{\partial}{\partial z_i}, z_i^2\frac{\partial}{\partial z_i}\right] = (2-s)z_i^{s + 1}\frac{\partial}{\partial z_i}.$$
        Since we already know that $z_i^s\frac{\partial}{\partial z_i}\in L$ for all $s\leq 3$ and since $2 - s \ne 0$ for $s\geq 3$, this implies by induction that $z_i^s\frac{\partial}{\partial z_i}\in L$ for all non-negative integers $s$.
        \item For any $i,j$ such that $i \ne j$ and for any non-negative integer $s$, we have
        $$\left[z_i^s\frac{\partial}{\partial z_i}, z_i\frac{\partial}{\partial z_j}\right] = z_i^s\frac{\partial}{\partial z_j}.$$
        This shows that $z_i^s\frac{\partial}{\partial z_j} \in L$ for any $i,j$ from $1$ to $n$ and for all positive integers $s$ (for $i = j$, it was proved at Step 5). 
        \item Now, let us prove that $f\frac{\partial}{\partial z_i}\in L$ for any $i$ and for any monomial $f$ in $z_1, z_2, \ldots, z_n$. To simplify the notation, we prove this in the case $i = 1$; for other values of $i$, the proof is the same. We show by induction on $k$ that $f\frac{\partial}{\partial z_1} \in L$ for any monomial $f$ containing only the variables $z_1, z_2, \ldots, z_k$ (maybe not all of them, i.e., some of these variables might have degree zero in $f$). For $k = 1$, it was proved at Step 5. To prove the induction step, suppose that $f\frac{\partial}{\partial z_1}\in L$ for any monomial $f$ depending only on the variables $z_1, z_2, \ldots, z_{k-1}$, where $2\leq k\leq n$. We have the following Lie bracket:
        $$\left[z_k^s\frac{\partial}{\partial z_1}, f\cdot z_1\frac{\partial}{\partial z_1}\right] = (p + 1)f\cdot z_k^s\frac{\partial}{\partial z_i},$$
        where $z_1$ is contained in $f$ with power $p$. This implies that $f\cdot z_k^s\frac{\partial}{\partial z_1}\in L$ for any monomial $f$ depending only on $z_1, z_2, \ldots, z_{k-1}$ and for any non-negative integer $s$. So, the proof of the induction step is finished.

        Since any element of $W_n$ is a linear combination of elements of the form $f\frac{\partial}{\partial z_i}$, where $f$ is a monomial in $z_1, z_2, \ldots, z_n$ and $1 \leq i \leq n$, the theorem is proved.
    \end{enumerate}
\end{proof}

However, the question whether the algebra $W_n$ is $1.5$-generated or not remains open.

\begin{problem}
    Is it true that for any nonzero $x\in W_n$ there exists $y\in W_n$ such that $x,y$ generate $W_n$ as a Lie algebra?
\end{problem}

\section{Complete vector fields}

In this section, suppose that $\KK$ is the field of complex numbers $\CC$. In this case, for any smooth affine algebraic variety $X$ and for any holomorphic (in particular, algebraic) vector field~$W$ on $X$, there is a classical notion of the flow map $\varphi_{W}$ of $W$ which we recall briefly. For $t\in \CC$ and $p\in X$, the value $\varphi_{W}(t,p) = \varphi_{W,t}(p)$ is defined as $\gamma(t)$, where $\gamma$ is the integral curve of $W$ (i.e., the map $\gamma\colon \Omega \to X$, where $\Omega\subseteq \CC$ is a neighbourhood of $0$ containing $t$, such that the tangent vector $\dot{\gamma}(s)$ is equal to the value of $W$ at the point $\gamma(s)$ for all $s\in \Omega$) with $\gamma(0) = p$. It follows from the existence and uniqueness theorem for ordinary differential equations that for any $p\in X$ the flow map $\varphi_{W,t}$ is defined for all sufficiently small values of~$t$. Moreover, the flow map is holomorphic and depends holomorphically on $t$.

There is the following notion of a complete vector field.

\begin{definition}
    An algebraic vector field on a complex affine algebraic variety $X$ is called \emph{complete} (or \emph{completely integrable}) if its flow map exists for all complex times. 
\end{definition}

For a complete vector field $W$ on $X$, the map $\varphi_{W,t}$ is a holomorphic automorphism of~$X$ for any $t\in\CC$. So, the flow map $\varphi_W$ defines a one-parameter subgroup in the group of holomorphic automorphisms of $X$. Note that the flow map of a complete vector field $W$ is not necessarily algebraic even if $W$ itself is algebraic. For example, the flow maps of $W_1 = z_1\frac{\partial}{\partial z_1}$ and $W_2 = z_1z_2\frac{\partial}{\partial z_1}$ on $\CC$ and $\CC^2$ are given, respectively, by the formulas
$$\varphi_{W_1,t}(z_1) = \textnormal{exp}(t)z_1,$$
$$\varphi_{W_2,t}(z_1,z_2) = (\textnormal{exp}(tz_2)z_1, z_2).$$
The first flow map does not depend algebraically on $t$ (although it defines an automorphism of $\CC$ for any fixed $t\in\CC$), while the second one is not an algebraic automorphism of $\CC^2$ even for a fixed $t\in \CC$.
 
However, if a vector field $W$ on $X$ is a locally nilpotent derivation (i.e., for any function $f\in\KK[X]$ there exists a positive integer $n$ such that $W^n(f) = 0$), then $W$ is complete and its flow map is algebraic and defines a $\mathbb G_a$-subgroup in the group of algebraic automorphisms of~$X$ (see \cite[Chapter 1, Section 5]{Fr}).

Complete vector fields are closely related to the notion of infinite transitivity.

\begin{definition}
    A subgroup $G$ in the group of holomorphic (or algebraic) automorphisms of a complex affine variety $X$ acts \emph{infinitely transitively} on $X$ if for all positive integers $m$ and for any two $m$-tuples $(p_1,p_2, \ldots, p_m)$, $(q_1, q_2, \ldots, q_m)$ of points of $X$ such that $p_i \ne p_j$ and $q_i \ne q_j$ for $i \ne j$ there is an element $g\in G$ such that $g(p_i) = q_i$ for all $i = 1,2\ldots, m$.
\end{definition}

The study of infinite transitivity has developed significantly in recent years and many interesting results were obtained (see, for example, \cite{Arzh, AFKKZ}).

It turns out that, under certain natural conditions on $X$, one can find a group generated by a finite number of one-parameter subgroups in the group of holomorphic automorphisms of $X$ acting infinitely transitively on $X$. Namely, if the Lie algebra of algebraic vector fields is generated by a finite number of algebraic vector fields, then the (holomorphic) flow maps of these vector fields generate a group acting infinitely transitively on $X$. In particular, since the Lie algebra of algebraic vector fields on $\CC^n$ is generated by three complete vector fields, an infinitely transitive action on $\CC^n$ can be obtained from three one-parameter subgroups (see \cite{Andr}). For more details, see \cite{Andr, Andr2, Andr3}. This is why it is interesting to find a finite set not of arbitrary  generators of $\textnormal{Lie}(X)$ but of complete ones.

Let us return to the vector fields $U, V$ from Theorem \ref{main}. It is easy to see that $U = \frac{\partial}{\partial z_n}$ is complete since it is a locally nilpotent derivation. Explicitly, its flow map is given by
$$\varphi_{U,t}(z_1, z_2, \ldots, z_n) = (z_1, z_2,\ldots, z_n + t).$$
However, let us show that the vector field $V$ from Section 2 is not complete for $n = 2$. We replace $z_1,z_2$ with $x,y$, respectively. To compute the flow map of $V = y^{8}\frac{\partial}{\partial x} + x^4y^4\frac{\partial}{\partial y}$, we need to solve the following system of differential equation:
$$\begin{cases}
    \dot{x} = y^{8},\\
    \dot{y} = x^4y^4.
\end{cases}$$
To avoid bulky notation, we compute the flow map only at the point $(x_0,y_0) = (1,1)$. The two equations imply that
$\frac{dy}{dx} = \frac{x^4}{y^4}.$
Using separation of variables, we have
$$y^4dy = x^4dx, \quad \frac{y^5}{5} = \frac{x^5}{5} + C_1,$$
where $C_1\in \CC$ is some constant.
Since the point $(x_0, y_0) = (1,1)$ should satisfy the last equation, we have $C_1 = 0$ and $y = x$. Using the first equation of the initial system, we have
$$\dot x = x^{8}, \quad \frac{dx}{dt} = x^{8}, \quad \frac{dx}{x^{8}} = dt, \quad -\frac{1}{7x^7} = t + C_2.$$
Assuming that $(x_0, y_0) = (1,1)$ corresponds to the flow time $t_0 = 0$, we have $C_2 = -\frac{1}{7}$, so
$$-\frac{1}{7x^7} = t - \frac{1}{7}.$$
Since the left hand side of the last equation is never zero, we see that the flow map is not defined for $t = \frac{1}{7}$. This means that the vector field $V$ is not complete. So, the following question remains open.

\begin{problem}
    Can the Lie algebra of algebraic vector fields on $\CC^n$ be generated by two complete vector fields?
\end{problem}







\end{document}